\documentclass[leqno,11pt]{amsart}

\usepackage{amssymb}
\usepackage{amsmath}
\usepackage{enumerate}
\usepackage[pdftex]{graphicx}

\def\s{\mathbb{S}}
\def\h{\mathbb{H}}
\def\r{\mathbb{R}}
\def\z{\mathbb{Z}}
\def\q{\mathbb{Q}}
\def\c{\mathbb{C}}

\newcommand{\arcsinh}{\mathop{\rm arcsinh }\nolimits}

\newtheorem{definition}{Definition}
\newtheorem{remark}{Remark}
\newtheorem{theorem}{Theorem}
\newtheorem{proposition}{Proposition}
\newtheorem{corollary}{Corollary}

\begin{document}

\title[HSL self-similar solutions to MCF]{Hamiltonian stationary self-similar solutions for Lagrangian mean curvature flow
\\ in complex Euclidean plane}

\author{Ildefonso Castro}
\address{Departamento de Matem\'{a}ticas \\
Universidad de Ja\'{e}n \\
23071 Ja\'{e}n, SPAIN} \email{icastro@ujaen.es}

\author{Ana M.~Lerma}
\address{Departamento de Matem\'{a}ticas \\
Universidad de Ja\'{e}n \\
23071 Ja\'{e}n, SPAIN} \email{alerma@ujaen.es}

\thanks{Research partially supported by a MEC-Feder grant MTM2007-61775}



\date{}



\date{}

\begin{abstract}
We classify all Hamiltonian stationary Lagrangian surfaces in
complex Euclidean plane which are self-similar solutions of the
mean curvature flow.
\end{abstract}

\maketitle

\section{Introduction}
The mean curvature flow is an evolution process under which a
submanifold deforms in the direction of its mean curvature vector.
By the first variation formula, the mean curvature vector points
to the direction where the volume decreases most rapidly.
Classically it has been studied by several approaches (partial
differential equations, geometric measure theory, level sets or
numerical methods) and its possible applications in symplectic
topology and mirror symmetry are quite important.

There are very interesting results on regularity, global existence
and convergence of the mean curvature flow in several ambient
spaces. The flow for hypersurfaces in arbitrary Riemannian
manifolds is well understood whereas not so much is known when the
codimension is greater than one.  In higher codimension the mean
curvature $H$ is a vector whose direction we do not know how to
control, as a contrast with the codimension one case, where $H$ is
essentially a scalar function whose sign is preserved along the
flow. In the last few years, mean curvature flow of higher
codimension submanifolds has attracted special attention, mainly
when the initial submanifold is Lagrangian in complex Euclidean
space $\c^n$; see \cite{ChL}, \cite{GSSZ}, \cite{N1}, \cite{N2},
\cite{TY} and \cite{Wa} for example. One reason of this growing
interest is that the Lagrangian condition is preserved by mean
curvature flow (see \cite{S}). In addition, as the gradient flow
of the volume functional, the mean curvature flow seems to be a
potential approach to the construction of special Lagrangians,
which are volume minimizers that play a critical role in the
T-duality formulation of mirror symmetry \cite{SYZ}.

When the ambient space is Euclidean, the mean curvature flow is
the solution to a system of parabolic equations that can be
considered as the heat equation for submanifolds. In general, the
mean curvature flow fails to exist after a finite time. The
singularities are completely determined by the blow up of the
second fundamental form. In geometric flows, like mean curvature
flow or Ricci flow, the singularities often model on soliton
solutions. In this article we are interested in a class of special
solutions of the mean curvature flow that preserve the shape of
the evolving submanifolds: the self-similar solutions, for which
the evolution is a homothety. Eliminating the time variable, these
self-similar solutions reduce the parabolic equation to the
non-linear elliptic equation $H=a\phi^\perp$, where $\phi^\perp $
stands for the projection of the position vector $\phi $ onto the
normal space. When $a$ is a negative constant, the submanifolds
shrinks in finite time to a single point under the action of the
mean curvature flow remaining its shape unchanged. If $a$ is
positive, the submanifold will expand with the same shape and in
this case is necessarily non-compact. If $a$ vanishes, the
submanifold is minimal and stationary under the action of the
flow.

The study of this type of solutions is hoped to give a better
understanding of the flow at a singularity since by Huisken's
monotonicity formula \cite{Hu}, any central blow-up of a
finite-time singularity of the mean curvature flow is a
self-similar solution.

Examples of self-similar solutions for mean curvature flow in
$\c^n$ were first constructed by Anciaux \cite{A} in 2006. In
order to produce eternal solutions of the Brakke flow (a weak
formulation of the mean curvature flow), Lee and Wang
\cite{LW1}-\cite{LW2} constructed in 2007 interesting examples of
Hamiltonian stationary self-shrinkers and self-expanders for
Lagrangian mean curvature flows, which are asymptotic to
Hamiltonian stationary cones generalizing Schoen-Wolfson ones
\cite{SW}. All of them have been generalized very recently by
Joyce, Lee and Tsui in \cite{JLT} providing examples with
different topologies.

It is expectable that the understanding of the singularities of
the mean curvature flow will rely on the classification of
self-similar solutions. But it is a hard and open problem classify
all the self-similar solutions for the mean curvature flow.  Using
strongly techniques of complex analysis, our contribution in this
paper is the classification of those ones in complex Euclidean
plane that are, in addition, Hamiltonian stationary Lagrangians,
i.e.\ critical points of the area functional among all Hamiltonian
deformations. The main result (Theorem 1) characterizes in this
way three one-parameter families of examples which include Lee and
Wang examples in dimension two (see Remarks 2, 3 and 5). We
provide (see Propositions 2 and 3) not only Hamiltonian stationary
Lagrangian conformal immersions of cylinders, Moebius strips, tori
and Klein bottles but also embedded Lagrangian nontrivial planes
all of them self-similar solutions for mean curvature flow. Our
construction is based essentially in a method given in \cite{CCh}
for obtaining Lagrangian surfaces in $\c^2$ starting from
spherical and hyperbolic curves; in this setting our examples are
constructed using simple particular geodesics and parallels (see
Remarks 4 and 6).

{\it Acknowledgments:} The authors would like to thank  Francisco
Urbano for his valuable suggestions and constant availability and
Henri Anciaux for helpful discussions.


\section{Hamiltonian stationary Lagrangian surfaces in $\c^2$}

In the complex Euclidean plane $\c^2$ we consider the bilinear Hermitian product defined by
\[
(z,w)=z_1\bar{w}_1+z_2\bar{w}_2, \quad  z,w\in\c^2.
\]
Then $\langle\, \, , \, \rangle = {\rm Re} (\,\, , \,)$ is the
Euclidean metric on $\c^2$ and $\omega = -{\rm Im} (\,,)$ is the
Kaehler two-form given by $\omega (\,\cdot\, ,\,\cdot\,)=\langle
J\cdot,\cdot\rangle$, where $J$ is the complex structure on
$\c^2$. We also have that $2 \omega = d\lambda$, where $\lambda$
is known as the Liouville 1-form of $\c^2$.

Let $\phi:M \rightarrow \c^2$ be an isometric immersion of a surface $M$ into $\c^2$. $\phi $ is
said to be Lagrangian if  $\phi^* \omega = 0$.  Then we have $\phi^* T\c^2 =\phi_* TM \oplus J
\phi_* T M$, where $TM$ is the tangent bundle of $M$. The second fundamental form $\sigma $ of
$\phi $ is given by $\sigma (v,w)=JA_{Jv}w$, where $A$ is the shape operator, and so the trilinear
form $C(\cdot,\cdot,\cdot)=\langle \sigma(\cdot,\cdot), J \cdot \rangle $ is fully symmetric.

Suppose $M$ is orientable and let $\omega_M$ be the area form of
$M$. If $\Omega = dz_1 \wedge dz_2$ is the closed complex-valued
2-form of $\c^2$, then $\phi^* \Omega = e^{i\beta} \omega_M$,
where $\beta:M\rightarrow \r /2\pi \z$ is called the {\em
Lagrangian angle} map of $\phi$ (see \cite{HL}).  In general
$\beta $ is a multivalued function; nevertheless $d\beta $ is a
well defined closed 1-form on $M$ and its cohomology class is
called the {\em Maslov class}. When $\beta $ is a single valued
function the Lagrangian is called {\em zero-Maslov class} and if
$\cos\beta \geq \epsilon$ for some $\epsilon >0$ then the
Lagrangian is said to be {\em almost calibrated}.

Note that the Lagrangian condition implies that the Liouville
1-form $\lambda $ is a closed 1-form on $M$. A Lagrangian
submanifold is said to be {\em monotone} if $[\lambda]=c[d\beta]$,
for some positive constant $c$. The standard examples of monotone
Lagrangians in $\c^2$ are the Clifford tori
\[
T_r\equiv \{ (z_1,z_2)\in \c^2 \,:\, |z_1|=|z_2|=r \}, \, r>0.
\]

It is remarkable that $\beta$ satisfies (see for example
\cite{SW})
\begin{equation}\label{beta}
J\nabla\beta=2H,
\end{equation}
where $H$ is the mean curvature vector of $\phi$, defined by
$H=(1/2) {\rm trace} \, \sigma$. (We point out that some other
authors, e.g.\ \cite{LW1}, \cite{LW2}, consider $H= {\rm trace} \,
\sigma$ and then $J\nabla\beta=H$).

If $\beta $ is constant, say $\beta\equiv \beta_0$ (or,
equivalently $H=0$), then the Lagrangian immersion $\phi $ is
calibrated by ${\rm Re} (e^{-i\beta_0}\Omega)$ and hence
area-minimizing. They are referred as being {\em Special
Lagrangian}.

A Lagrangian submanifold is called {\em Hamiltonian stationary} if
the Lagrangian angle $\beta $ is harmonic, i.e. $\Delta \beta =0$,
where $\Delta $ is the Laplace operator on $M$. Hamiltonian
stationary Lagrangian (in short HSL) surfaces are critical points
of the area functional with respect to a special class of
infinitesimal variations preserving the Lagrangian constraint;
namely, the class of compactly supported Hamiltonian vector fields
(see \cite{O}). Examples of HSL surfaces in $\c^2$ can be found in
\cite{AC}, \cite{CU2} and \cite{HR1}.

\section{Self-similar solutions for the mean curvature flow}

Let $\phi: M\rightarrow \r^4$ be an immersion of a smooth surface $M$ in Euclidean 4-space. It is
not hard to check that
\begin{equation}\label{Hdelta}
2H=\Delta \phi,
\end{equation}
where $H$ is the mean curvature vector and $\Delta $ is the Laplace operator of the induced metric
on $M$.

The mean curvature flow (in short MCF) of $\phi: M\rightarrow
\r^4$ is a family of immersions $F:M \times [0,T_{\mathrm{sing}})
\rightarrow \r^4$ parameterized by $t$ that satisfies
\begin{equation}\label{MCF}
\frac{d}{dt}F_t(p)=H(p,t), \quad F_0=\phi,
\end{equation}
where $H(p,t)$ is the mean curvature vector of $F_t(M)$ at
$F_t(p)=F(p,t)$ and $[0,T_{\mathrm{sing}})$ is the maximal time
interval such that (\ref{MCF}) holds. Looking at equation
(\ref{Hdelta}), this can be considered as the heat equation for
submanifolds.

The equation (\ref{MCF}) is a quasi-linear parabolic system and short time existence is guaranteed
when the initial submanifold $M$ is compact.

When $\phi: M\rightarrow \r^4\equiv \c^2$ is Lagrangian it is well
known  that $0 < T_{\mathrm{sing}}< \infty$ and being Lagrangian
is preserved along the mean curvature flow. The monotone condition
defined in Section 2 is also preserved under the flow (see
\cite{GSSZ}).

In geometric flows such as the Ricci flow or the MCF, singularities are often locally modelled on
soliton solutions. In the case of MCF, one type of soliton solutions of great interest are those
moved by scaling in the Euclidean space. We recall that solitons moved by scaling must be of the
following form:

\begin{definition}
An immersion $\phi: M\rightarrow \r^4$ is called a self-similar solution for mean curvature flow if
\begin{equation}\label{ss}
H=a\, \phi^\perp
\end{equation}
for some nonzero constant a, where $\phi^\perp $ denotes the normal projection of the position
vector $\phi $ and $H$ is the mean curvature vector of $\phi $. If $a<0$, it is called a
self-shrinker and if $a>0$ it is called a self-expander.
\end{definition}

\begin{remark}
{\rm  In the Lagrangian setting, using (\ref{beta}) the condition
(\ref{ss}) to be a self-similar for MCF can be reformulated in
terms of the Lagrangian angle map and the restriction to the
surface of the Liouville 1-form by means of $$ d\beta=-4a \phi^*
\lambda, \ a\neq 0. $$ In particular, any Lagrangian self-similar
solution for MCF is monotone.}
\end{remark}

If $\phi $ is a self-similar solution, then
$F_t:=\sqrt{2at+1}\,\phi $, $2at+1>0$, is a solution of
(\ref{MCF}). When $a=0$ the submanifold is minimal, i.e. $H=0$,
and is fixed by MCF.

One could normalize the value of $a$ up to dilations, but we
prefer not to do it because we will take different particular
values of $a$ along the paper to recover some known examples.

It is an exercise to check that, given any $a<0$, the right
circular cylinder $\s^1({1 \over \sqrt{-2a}})\times \r$ and the
Clifford torus $\s^1({1\over \sqrt{-2a}})\times
\s^1({1\over\sqrt{-2a}})$ satisfy equation (\ref{ss}). Thus they
are examples of self-shrinkers. Both of them are also Lagrangian
surfaces with parallel mean curvature vector and hence Hamiltonian
stationary Lagrangians.

We finish this section studying a close relationship between the
area and the Willmore functional of compact Lagrangian
self-shrinkers, that it is even true for monotone compact HSL
surfaces (see \cite{Mi}).

\begin{proposition}\label{areaWillmore}
Let $\phi: M\rightarrow \c^2$ be a Lagrangian self-similar
solution for mean curvature flow, i.e.\ $H=a\phi^\perp$, $a\neq
0$. Then $$\mathrm{div}\,\phi^\top = 2\left( 1+\frac{|H|^2}{a}
\right),
$$ where $\phi^\top$ denotes the tangent projection of the position vector $\phi$. In
particular, if $M$ is compact then $a<0$ and
$\mathrm{Area}(M)=-\frac{1}{a}\int_M |H|^2$.
\end{proposition}
\begin{proof}
Derivating in the direction of any tangent vector $v$ in the
equality $\phi = \phi^\top + H/a$ and taking tangent components,
we get that $v=\nabla_v \phi^\top - \frac{1}{a}A_H v$. If $\{ e_1,
e_2 \}$ is an orthonormal frame in $M$, then
$\mathrm{div}\,\phi^\top = 2 +\frac{1}{a} \sum_{i=1}^2 \langle A_H
e_i, e_i \rangle =2 + \frac{2}{a}|H|^2.$
\end{proof}


\section{The examples} 

In this section we introduce three one-parameter families of HSL
surfaces in $\c^2$ which are self-similar solutions for mean
curvature flow satisfying (\ref{ss}) and describe their main
geometric properties. In the next result, we provide examples with
the topology of a plane, a cylinder or a Moebius strip.

\begin{proposition}\label{ex1}
Given any $a > 0$, let define:
$$ \Phi_\delta:\r^2 \rightarrow \c^2, \ \delta >0,$$
\begin{equation}\label{Phib}
\Phi_\delta(s,t)= \frac{1}{\sqrt{2a}} \left( i \,\mathrm{s}_\delta
\,\cosh t\,e^{-\frac{i\,s}{\mathrm{c}_\delta}} , \mathrm{t}_\delta
\, \sinh t\, e^{i\, \mathrm{c}_\delta\,s} \right),
\end{equation}
with $\mathrm{s}_\delta=\sinh \delta $, $\mathrm{c}_\delta=\cosh
\delta$ and $\mathrm{t}_\delta=\tanh \delta$. Then $\Phi_\delta$
is a Hamiltonian stationary Lagrangian conformal immersion and a
self-similar solution for mean curvature flow satisfying
(\ref{ss}).

If $\cosh^2 \delta \notin \q$, $\Phi_\delta$ is  -in addition- an
embedded self-expander plane asymptotic to the Hamiltonian
stationary Lagrangian cone
$$\{ (i\, x_1\, e^{-\frac{i\,s}{\mathrm{c}_\delta}},x_2\,  e^{i\,
\mathrm{c}_\delta\,s}) \, : \, (x_1,x_2)\in \r^2, \, s\in \r , \,
x_1^2=\cosh^2 \delta \, x_2^2 \}.$$

If $\cosh^2 \delta = p/q \in \q$, $(p,q)=1$, $p>q$, then
$\Phi_\delta$ is given by 
$$\Phi_{p,q}:\r^2 \rightarrow \c^2, \ p>q,$$
\begin{equation}\label{Phipq}
\Phi_{p,q}(s,t)=\sqrt{\frac{p-q}{2a}} \left(
\frac{i}{\sqrt{q}}\,\cosh t\,e^{-i \sqrt{\frac{q}{p}}s} ,
\frac{1}{\sqrt{p}}\,\sinh t\,e^{i \sqrt{\frac{p}{q}}s} \right)
\end{equation}
and satisfies $\Phi_{p,q}(s+2\pi\sqrt{pq},t)=\Phi_{p,q}(s,t)$,
$\forall (s,t)\in \r^2$, inducing a Hamiltonian stationary
Lagrangian immersion of a self-expander cylinder.

Moreover, if $p$ is odd and $q$ is even, then
$\Phi_{p,q}(s+\pi\sqrt{pq},-t)=\Phi_{p,q}(s,t)$, $\forall (s,t)\in
\r^2$, inducing a Hamiltonian stationary Lagrangian immersion of a
self-expander Moebius strip.
\vspace{0.5cm}

 Given any $a < 0$, let define:
$$\Upsilon_\gamma:\r^2 \rightarrow \c^2, \  0<\gamma <\pi/2,$$
\begin{equation}\label{Phib2}
\Upsilon_\gamma(s,t)= \frac{1}{\sqrt{-2a}} \left( -i \,
\mathrm{s}_\gamma \,\cosh t\,e^{\frac{i\,s}{\mathrm{c}_\gamma}} ,
\mathrm{t}_\gamma \, \sinh t\,e^{-i\, \mathrm{c}_\gamma\,s}
\right),
\end{equation}
with $\mathrm{s}_\gamma=\sin \gamma $, $\mathrm{c}_\gamma=\cos
\gamma$ and $\mathrm{t}_\gamma=\tan \gamma$. Then
$\Upsilon_\gamma$ is a Hamiltonian stationary Lagrangian conformal
immersion and a self-similar solution for mean curvature flow
satisfying (\ref{ss}).

If $\cos^2 \gamma \notin \q$, $\Upsilon_\gamma$ is -in addition-
an embedded self-shrinker plane asymptotic to the Hamiltonian
stationary Lagrangian cone
$$\{ (-i\, x_1\, e^{\frac{i\,s}{\mathrm{c}_\gamma}}, x_2\, e^{-i\, \mathrm{c}_\gamma\,s})
\, : \, (x_1,x_2)\in \r^2, \, s\in \r, \,  x_1^2=\cos^2 \gamma \,
x_2^2\}.$$

If $\cos^2 \gamma = p/q \in \q$, $(p,q)=1$, $p<q$, then
$\Upsilon_\gamma$ is given by 
$$\Upsilon_{p,q}: \r^2 \rightarrow \c^2, \
p<q,$$
\begin{equation}\label{Phipq2}
\Upsilon_{p,q}(s,t)=\sqrt{\frac{q-p}{-2a}} \left(
\frac{-i}{\sqrt{q}}\,\cosh t\,e^{i \sqrt{\frac{q}{p}}s} ,
\frac{1}{\sqrt{p}}\,\sinh t\,e^{-i \sqrt{\frac{p}{q}}s} \right)
\end{equation}
and satisfies
$\Upsilon_{p,q}(s+2\pi\sqrt{pq},t)=\Upsilon_{p,q}(s,t)$, $\forall
(s,t)\in \r^2$, inducing a Hamiltonian stationary Lagrangian
immersion of a self-shrinker cylinder.

Moreover, if $q$ is even and $p$ is odd, then
$\Upsilon_{p,q}(s+\pi\sqrt{pq},-t)=\Upsilon_{p,q}(s,t)$, $\forall
(s,t)\in \r^2$, inducing a Hamiltonian stationary Lagrangian
immersion of a self-shrinker Moebius strip.
\end{proposition}

\begin{remark}
{\rm We can recover the Hamiltonian stationary expanders
$\mathcal{E}$ and the Hamiltonian stationary shrinkers
$\mathcal{S}$ studied in Proposition 2.1 of \cite{LW1} by Lee and
Wang taking
$\mathcal{E}(u,\theta)=\Phi_{p,q}(\sqrt{pq}\,\theta,u)$ and
$\mathcal{S}(u,\theta)=i\Upsilon_{p,q}(\sqrt{pq}\,\theta,u)$ with
$\Phi_{p,q}$ given respectively in (\ref{Phipq}) and
$\Upsilon_{p,q}$ in (\ref{Phipq2}) being $a=\frac{p-q}{2pq}$. }
\end{remark}
\begin{remark}
{\rm By taking $x_1=\sqrt{\frac{p-q}{2a}}\,\frac{\cosh
t}{\sqrt{q}} >0$ and $x_2=\sqrt{\frac{p-q}{2a}}\,\frac{\sinh
t}{\sqrt{p}}$ we can rewrite
$$ \left( \begin{array}{cc} -i
& 0 \\ 0 & 1 \end{array} \right)\Phi_{p,q}(\sqrt{pq}\,\theta,t)
=(x_1 e^{-iq\theta},x_2 e^{ip\theta}), \ -q\,x_1^2 + p \,x_2^2 =
\frac{q-p}{2a}<0 $$ and
$$ \left( \begin{array}{cc} i
& 0 \\ 0 & 1 \end{array}
\right)\Upsilon_{p,q}(\sqrt{pq}\,\theta,t) =(x_1 e^{iq\theta},x_2
e^{-ip\theta}), \ q\,x_1^2 - p \,x_2^2 = \frac{q-p}{-2a}>0, $$
where $ 0\leq \theta < 2\pi $. In this way we arrive at the
examples given in Proposition 2.1 of \cite{LW2} for $n=2$.
Restricting $\theta \in [0,\pi)$, in Proposition 2.2 of \cite{LW2}
it is proved that these examples are oriented if and only if $p-q$
is even and are embedded if and only $q=1$.

We finally remark that $\Phi_\delta$ and $\Upsilon_\gamma$
generalize them although they can be included in the observation
made in Remark 2.1 of \cite{LW2}.}
\end{remark}

\begin{remark}
{\rm Following the spirit of \cite{CCh}, the Lagrangians
$\Phi_\delta $ (resp.\ $\Upsilon_\gamma$) are constructed with the
Legendrian geodesic $t\rightarrow (\sinh t, \cosh t)$ in the anti
De Sitter space $\h^1_3$ and the constant curvature Legendrian
curves $s \rightarrow \frac{1}{\sqrt{2a}}(\mathrm{t}_\delta \,
e^{i\, \mathrm{c}_\delta\,s} ,i \,\mathrm{s}_\delta \,e^{-i\,s /
\mathrm{c}_\delta})$ (resp.\ $s \rightarrow
\frac{1}{\sqrt{2a}}(\mathrm{t}_\gamma \, e^{-i\,
\mathrm{c}_\gamma\,s},-i \, \mathrm{s}_\gamma \,e^{i\,s /
\mathrm{c}_\gamma})$ ) in 3-spheres.}
\end{remark}

\begin{proof}
Given any $a>0$, let $\Phi=\Phi_\delta$. It is an exercise to
check that
$$
|\Phi_s|^2=\frac{1}{2a}\left(
\mathrm{t}_\delta^2\cosh^2t+\mathrm{s}_\delta^2
\sinh^2t\right)=|\Phi_t|^2, \ (\Phi_s,\Phi_t)=0.
$$
This shows that $\Phi$ is a Lagrangian conformal immersion whose
induced metric is given by $\langle ,
\rangle=e^{2u(t)}\left(ds^2+dt^2\right)$, with $e^{2u(t)}
=\frac{1}{2a}\left(
\mathrm{t}_\delta^2\cosh^2t+\mathrm{s}_\delta^2 \sinh^2t\right)$.
Moreover we compute
$e^{i\beta_{\Phi}}=e^{-2u}\textstyle\det_\c\displaystyle(\Phi_s,\Phi_t)=
e^{i\,\mathrm{s}_\delta^2 \,s / \mathrm{c}_\delta}$. So we
conclude that $\Phi $ is HSL since $\beta_\Phi$ depends only on
$s$. Using (\ref{beta}) it is not hard to obtain that the mean
curvature vector of $\Phi$ satisfies that $
H_\Phi=\frac{\mathrm{s}_\delta^2
e^{-2u}}{2\mathrm{c}_\delta}J\Phi_s=a \Phi^\perp $.

To prove that $\Phi $ is an embedding, we start from
$\Phi(s_1,t_1)=\Phi(s_2,t_2)$ and then it is clear that when
$(s_1,t_1)\neq (s_2,t_2)$  necessarily $\mathrm{c}_\delta^2$ is a
rational number.

Finally, $\Phi$ is asymptotic to the given HSL cone taking into
account that if $t\rightarrow + \infty$ then  $\cosh t, \sinh t
\simeq e^t /2 $ and if $t\rightarrow - \infty$ then $\cosh t
\simeq  e^{-t} /2 $ and $\sinh t \simeq - e^{-t} /2 $.

The reasoning for $\Upsilon_\gamma$ is completely similar so that
we omit it.
\end{proof}

The next result describes examples with the topology of a
cylinder, a torus or a Klein bottle.

\begin{proposition}\label{ex2}
Given any $a < 0$, let define:
$$\Psi_\nu:\s^1 \times \r \rightarrow \c^2, \ \nu >0,$$
\begin{equation}\label{Phic}
\Psi_\nu(e^{i\,s},t)= \frac{1}{\sqrt{-2a}} \left( \mathrm{c}_\nu\,
\cos s\,e^{\frac{i\,t}{\mathrm{s}_\nu}},\mathrm{t}_\nu \, \sin
s\,e^{i\, \mathrm{s}_\nu \,t}\right),
\end{equation}
with $\mathrm{s}_\nu=\sinh \nu $, $\mathrm{c}_\nu=\cosh \nu$ and
$\mathrm{t}_\nu=\coth \nu$. Then $\Psi_\nu$ is a Hamiltonian
stationary Lagrangian conformal immersion and a self-similar
solution for mean curvature flow satisfying (\ref{ss}). If
$\sinh^2 \nu \notin \q$, $\Psi_\nu$ is -in addition- an embedded
self-shrinker cylinder.

If $\sinh^2 \nu = m/n \in \q$, $(m,n)=1$, then
$\Psi_\nu$ is given by 
$$\Psi_{m,n}:\s^1 \times \r \rightarrow \c^2, \ (m,n)=1, $$
\begin{equation}\label{Phimn}
\Psi_{m,n}(s,t)=\sqrt{\frac{m+n}{-2a}} \left(
\frac{1}{\sqrt{n}}\,\cos s\,e^{i \sqrt{\frac{n}{m}}t} ,
\frac{1}{\sqrt{m}}\,\sin s\,e^{i \sqrt{\frac{m}{n}}t} \right).
\end{equation}
Moreover, $\Psi\equiv\Psi_{m,n}$ satisfy the following properties:
\begin{enumerate}
\item $\Psi (s+2\pi, t)=\Psi (s,t) =\Psi (s,t+2\pi\sqrt{mn})
$, $\forall (s,t)\in\r^2$; if, in addition, $m$ and $n$ are odd
then $\Psi (s+\pi,t+\pi\sqrt{mn})=\Psi (s,t)$, $\forall
(s,t)\in\r^2$. Let $\Lambda_{m,n}$ the lattice spanned by
$(2\pi,0)$ and $(0,2\pi\sqrt{mn})$ (resp.\ $(2\pi,0)$ and
$(\pi,\pi\sqrt{mn})$) when $m$ or $n$ are even (resp.\ when $m$
and $n$ are odd) and $\mathcal{T}_{m,n}=\r^2 / \Lambda_{m,n}$ the
corresponding self-shrinker torus. Then:
$$ \mathrm{Area}(\mathcal{T}_{m,n})=\left\{ \begin{array}{ll}
\frac{(m+n)^2 \pi^2}{-a \sqrt{mn}}, & m \mathrm{\ or \ } n  \mathrm{\ even} \\ \\
\frac{(m+n)^2 \pi^2}{-2a \sqrt{mn}}, & m \mathrm{\ and \ } n
\mathrm{\ odd}
\end{array}\right.
$$
\item If $m$ is odd and $n$ is even, then $\Psi (2\pi -s,t+\pi\sqrt{mn})=\Psi (s,t)$,
$\forall (s,t)\in\r^2$.

If $m$ is even and $n$ is odd, then $\Psi (\pi
-s,t+\pi\sqrt{mn})=\Psi (s,t)$,
$\forall (s,t)\in\r^2$.

In both cases, $m+n$ is odd and $\mathcal{T}_{m,n}$ is the
covering of the corresponding self-shrinker Klein bottle
$\mathcal{K}_{m,n}$.
\item The Clifford torus $\mathcal{T}_{1,1} $ is the only one embedded in the above family.
\end{enumerate}
\end{proposition}

\begin{remark}
{\rm The immersion $\Psi_{\arcsinh 1}\equiv\Psi_{1,1}$
corresponding to the Clifford torus $T_{1/\sqrt{-2a}}$, $a<0$, can
be checked to be the only Willmore surface in this family. Up to
isometries, it is enough to consider $\nu \in (0,\arcsinh 1] $
(and hence $m\leq n$) because $\Psi_{\hat \nu}$, with $ \hat{\nu}=
\log (\coth \nu/2)$, is congruent to $\Psi_\nu$.

By taking $x_1=\sqrt{\frac{m+n}{-2a}}\,\frac{\cos s}{\sqrt{n}}$
and $x_2=\sqrt{\frac{m+n}{-2a}}\,\frac{\sin s}{\sqrt{m}}$, we can
rewrite $$\Psi_{m,n}(s,\sqrt{mn}\,\theta)=( x_1 e^{in\theta},x_2
e^{i m\theta}), \ n x_1^2 + m x_2^2 = C=\frac{m+n}{-2a}>0.$$
Therefore we arrive again at other examples considered in
Proposition 2.1 of \cite{LW2} by Lee and Wang when $n=2$. We
remark that $\Psi_\nu$ generalize them although they can be
included in the observation made in Remark 2.1 of \cite{LW2}.

Using Proposition \ref{areaWillmore} we also get that the Willmore
functional $\mathcal{W}:=\int_M |H|^2 $ of the tori
$\mathcal{T}_{m,n}$ is given by
$$ \mathcal{W}(\mathcal{T}_{m,n})=\left\{ \begin{array}{ll}
\frac{(m+n)^2 \pi^2}{ \sqrt{mn}}, & m \mathrm{\ or \ } n  \mathrm{\ even} \\ \\
\frac{(m+n)^2 \pi^2}{2 \sqrt{mn}}, & m \mathrm{\ and \ } n
\mathrm{\ odd}
\end{array}\right.
$$}
\end{remark}

\begin{remark}
{\rm Following the spirit of \cite{CCh}, the Lagrangians $\Psi_\nu
$ are constructed with the constant curvature Legendrian curves
$t\rightarrow \frac{1}{\sqrt{-2a}}( \mathrm{c}_\nu\, e^{i\,t /
\mathrm{s}_\nu},\mathrm{t}_\nu \,e^{i\, \mathrm{s}_\nu \,t})$ in
anti De Sitter spaces and the Legendrian geodesic $s \rightarrow
(\cos s, \sin s)$ in the unit 3-sphere.

On the other hand, it is clear that the HSL tori
$\mathcal{T}_{m,n}$ admit a one-parameter group of isometries.
Using the notation of \cite{CU2}, it is not complicated to check
that their universal coverings would correspond to the immersions
$\mathcal{F}_{0,\beta}^{\pi/2}$, $\sin \beta \in \q$.}
\end{remark}

\begin{proof}
Given any $a<0$, let $\Psi=\Psi_\nu$. It is an exercise to check
that
$$
|\Psi_s|^2=\frac{1}{-2a}\left( \mathrm{t}_\nu^2\cos^2 s
+\mathrm{c}_\nu^2 \sin^2 s\right)=|\Psi_t|^2, \ (\Psi_s,\Psi_t)=0.
$$
This shows that $\Psi$ is a Lagrangian conformal immersion whose
induced metric is given by $\langle ,
\rangle=e^{2v(s)}\left(ds^2+dt^2\right)$, with $e^{2v(s)}
=\frac{1}{-2a}\left( \mathrm{t}_\nu^2\cos^2 s +\mathrm{c}_\nu^2
\sin^2 s\right)$. Moreover we get that $e^{i\beta_{\Psi}(t)}=
-i\,e^{i\,\mathrm{c}_\nu^2 \,t / \mathrm{s}_\nu}$. Thus we
conclude that $\Psi $ is HSL. Using (\ref{beta}) it is easy to
obtain that the mean curvature vector of $\Psi$ satisfies that $
H_\Psi=\frac{\mathrm{c}_\nu^2 e^{-2v}}{2\mathrm{s}_\nu}J\Psi_t=a
\Psi^\perp $.

It is straightforward to get that $\Psi $ is an embedding since
$\Psi(s_1,t_1)=\Psi(s_2,t_2)$ and $(s_1,t_1)\neq (s_2,t_2)$
implies that $\mathrm{s}_\nu^2$ is a rational number.

The geometric properties of $\Psi_{m,n}$ are deduced making use of
the above data and Remark 4.
\end{proof}


\section{Classification}

\begin{theorem}\label{main}
Let $\phi: M^2 \rightarrow \c^2 $ be a Hamiltonian stationary Lagrangian self-similar solution for
mean curvature flow.
\begin{itemize}
\item[(a)] If $\phi$ is a self-expander, i.e.\ $H=a\phi^\perp$, $a>0$, then $\phi $ is locally congruent
to some $\Phi_\delta:\r^2 \rightarrow \c^2$, $\delta>0$, described
in Proposition \ref{ex1}.
\item[(b)] If $\phi$ is a self-shrinker, i.e.\ $H=a\phi^\perp$, $a<0$, then $\phi $ is locally congruent
to some of the following:
\begin{itemize}
\item[(i)] the right circular cylinder $\s^1({1 \over \sqrt{-2a}})\times \r$;
\item[(ii)] the Clifford torus $\s^1({1\over \sqrt{-2a}})\times \s^1({1\over\sqrt{-2a}})$;
\item[(iii)] some $\Upsilon_\gamma:\r^2 \rightarrow \c^2$, $0<\gamma<\pi/2 $, described in Proposition \ref{ex1};
\item[(iv)] some $\Psi_\nu:\s^1 \times \r \rightarrow \c^2$, $\nu >0 $, described in Proposition \ref{ex2}.
\end{itemize}
\end{itemize}
\end{theorem}

\begin{proof}
First, following \cite{CU1} we can associate to any Lagrangian
immersion $\phi :M \longrightarrow \c ^2$ a cubic differential
form $\Theta $ on $M$ defined by
\[ \Theta (z)=f(z)(dz)^3, {\rm \ with \ } f(z)=4C(\partial_{z},\partial_{z},
  \partial_{z}) \]
and a differential form $\Lambda $ on $M$ defined by
\[ \Lambda (z)=\bar{h}(z)dz, {\rm \ with \ } h(z)=2 \omega (
\partial_{\bar z}, H ), \]
where $ z=x+iy$ is a local isothermal coordinate on $M$ such that
the induced metric, also denoted by $\langle \,\, , \, \rangle $,
is written as $\langle ,\rangle = e^{2u}|dz|^2$ with $|dz|^2$ the
Euclidean metric, and $C$ and $\omega $ are extended $\c$-linearly
to the complexified tangent bundles. We remark that our $h$ here
is $\bar h$ in \cite{CU1}.

It is straightforward to check that the Frenet equations of $\phi $  are given by
\begin{eqnarray}\label{Frenet}
 \phi_{zz} = 2u_{z}\phi_{z} + \frac{\bar h}{2} J\phi_{z}+
 \frac{e^{-2u} f}{2} J\phi_{\overline{z}},  \\
  \phi_{z \overline{z}} = \frac{h}{2}
              J\phi_{z}+\frac{\overline h}{2}J\phi_{\overline{z}}, \nonumber
\end{eqnarray}
and it is not difficult (see equation (3.3) in \cite{CU1}) to get
the compatibility equations for (\ref{Frenet}), obtaining
\begin{eqnarray}\label{compa}
4u_{z  \overline{z}}+\frac{|h|^2-e^{-4u}|f|^2}{2}  =  0 \nonumber \\
{\rm Im}(h_z)  =  0 \\
\overline{f_{\overline{z}}}=e^{2u}(h_{\bar z}-2u_{\bar z}h) \nonumber
\end{eqnarray}
Now (\ref{beta}) translates into $h=\beta_{\bar z}$.

Since $\phi $ is a self-similar solution for mean curvature flow,
using (\ref{ss}) we have that $h=-2a\langle \phi_{\overline z},
J\phi \rangle$, $a\neq 0$, and from (\ref{Frenet}) we deduce that
\begin{equation}\label{hz}
h_z=-a {\rm Re}(h|\phi|^2_z),
\end{equation}
and, taking into account that $\langle \phi_z , \phi_{\bar z}
\rangle=e^{2u}/2$ and $\langle \phi_z, \phi_z \rangle =0$, from
(\ref{Frenet}) we also have that
\begin{equation}\label{mod}
|\phi|^2_{z \overline{z}}=\frac{|h|^2}{a}+e^{2u}
\end{equation}
and
\begin{equation}\label{f}
|\phi|^2_{zz}= 2u_z |\phi|^2_z + \frac{1}{2a} \left( \bar{h}^2 +
e^{-2u} f h \right).
\end{equation}

 As $\phi $ is also a Hamiltonian stationary Lagrangian
immersion, the second equation of (\ref{compa}) implies that
$\bar{h}_{\bar{z}}=h_z=\beta_{z \overline{z}}=0$. Hence $\Lambda$
is a holomorphic differential and we can normalize $h\equiv \mu$,
$\mu >0$, since the zeroes of $h$ and $H$ are the same and $a\neq
0$. Thus (\ref{hz}) says that $g:=|\phi|^2$ satisfies $g_x=0$,
that is $g=g(y)$. In addition, from (\ref{mod}) $g$ satisfies
\begin{equation}\label{g´´}
g''=4\left(\frac{\mu^2}{a}+e^{2u} \right).
\end{equation}
In particular, we get that $u=u(y)$ too. We can now express $f$ in
terms of $g$ and $u$ from (\ref{f}) by
\begin{equation}\label{fexp}
f=\frac{e^{2u}}{\mu}\left(a u'g'-\frac{a}{2}g''-\mu^2 \right).
\end{equation}

If $\phi^\top$ denotes the tangent part of $\phi$, using
$|h|^2=e^{2u}|H|^2$ and taking modules in the equality
$\phi=\phi^\top+H/a$ yields to
\begin{equation}\label{g}
g=e^{-2u}\left( \frac{g'^2}{4}+\frac{\mu^2}{a^2} \right)
\end{equation}
This implies that $g>0$.

From (\ref{g´´}) and (\ref{g}) we arrive at the following o.d.e.\ for $g$:
\begin{equation}\label{odeg}
a^2(g g'' - g'^2)=4\mu^2(1+a g).
\end{equation}

Only when $a<0$ the equation (\ref{odeg}) has a constant solution
$g\equiv -1/a$. In this case (\ref{g´´}) or (\ref{g}) gives
$e^{2u}\equiv -\mu^2/a$ and (\ref{fexp}) gives $f\equiv \mu^3/a$.
The integration of the corresponding Frenet equations
(\ref{Frenet}), now simply written as $\phi_{xx}=\phi_{yy}=\mu
J\phi_x, \ \phi_{xy}=\mu J\phi_y$, leads to the Clifford torus
$\s^1({1\over \sqrt{-2a}})\times \s^1({1\over\sqrt{-2a}})$. This
proves part (b)-(ii).

In the general case, we obtain a first integral for (\ref{odeg}) given by
\begin{equation}\label{energy}
g'^2=P(g):=4Eg^2-\frac{8\mu^2}{a}g-\frac{4\mu^2}{a^2}, \ \  E\in
\r.
\end{equation}
We now look for the o.d.e.\ for the conformal factor of the
induced metric. Using (\ref{g}) and (\ref{energy}) we have that
\begin{equation}\label{uyg}
e^{2u}=E g -2\mu^2/a .
\end{equation}
Then (\ref{energy}) translates into
\begin{equation}\label{energy2}
u'^2-{2\mu^2 E \over a}e^{-2u} + {\mu^2E^2 \over a^2}e^{-4u}=E
\end{equation}
and so
\begin{equation}\label{u´´}
u''+{2\mu^2 E \over a}e^{-2u}- {2\mu^2 E^2\over a^2} e^{-4u} =0.
\end{equation}
Using (\ref{uyg}) and (\ref{energy2}), (\ref{fexp}) implies that
\begin{equation}\label{fend}
f=\mu (e^{2u}-2E/a),
\end{equation}
that is compatible with (\ref{compa}) and (\ref{u´´}).

Then we can rewrite the Frenet equations (\ref{Frenet}) in the
following way:
\begin{eqnarray}\label{Frenetreal}
&& \phi_{xx}=-u' \phi_y +\left( 2\mu-\frac{\mu E}{a} e^{-2u} \right)J\phi_x\nonumber\\
&& \phi_{xy}=u' \phi_x + \frac{\mu E}{a} e^{-2u} J\phi_y\\
&& \phi_{yy}=u' \phi_y + \frac{\mu E}{a} e^{-2u} J\phi_x \nonumber
\end{eqnarray}
 After a long computation, using (\ref{Frenetreal}), (\ref{energy2}) and (\ref{u´´}),
we get that $ \phi_{xyy}=E \phi_x $ and $ \phi_{yyy}=E \phi_y $.
Up to translations, we can consider
\begin{equation}\label{Frenet1}
\phi_{yy}=E \phi
\end{equation}
and (\ref{Frenetreal}) gives
\begin{equation}\label{Frenet2}
\phi_{xx} = -E \phi + 2 \mu J\phi_x .
\end{equation}
In particular, $H=\mu e^{-2u}J\phi_x$.

On the one hand, when $E=0$ it is necessarily $a<0$ from
(\ref{energy}). Using (\ref{uyg}), we obtain that $e^{2u}\equiv
-2\mu^2/a$ and (\ref{fend}) gives $f\equiv -2\mu^3/a$. The
integration of the corresponding Frenet equations
(\ref{Frenetreal}), now simply written as $\phi_{xx}=2\mu J\phi_x,
\ \phi_{xy}=\phi_{yy}=0$ leads to the right circular cylinder
$\s^1({1\over\sqrt{-2a}})\times \r$. This proves part (b)-(i).

On the other hand, when $E\neq 0$, the discriminant of the second
order polynomial $P(g)$ in (\ref{energy}) is
$64\mu^2(\mu^2+E)/a^2$. As $P(g)=g'^2$ must be non negative for
$g>0$ (and observe that $P(0)<0$) we distinguish two cases
according to the sign of the energy $E$ to reach the following
conclusions:
\begin{itemize}
\item
{\em Case (a):} If $E>0$ then $g$ is bounded from below.
\item
{\em Case (b):} If $E<0$, it is necessarily $E\geq -\mu^2$, and
$g$ is bounded from below and from above. We remark that
(\ref{uyg}) shows that if $E<0$ then it is necessarily $a<0$ and
we also point out that the case $E=-\mu^2$ corresponds precisely
with the constant case $g\equiv1/-a$.
\end{itemize}

We now proceed to integrate explicitly (\ref{Frenetreal}) through
(\ref{Frenet1}) and (\ref{Frenet2}) . From (\ref{uyg}) there is no
restriction supposing $u'(0)=0$. Let $\alpha:=e^{2u(0)}>0$. So
(\ref{energy2}) says that
\begin{equation}\label{ctes}
E=a \alpha \left(2 +\frac{a\alpha}{\mu^2} \right).
\end{equation}

{\em Case (a): $E>0$.} Using (\ref{Frenetreal}), (\ref{Frenet1})
and (\ref{Frenet2}), we get
$$ \phi(x,y)= \cosh (\sqrt{E}y)C_1(x) + \sinh (\sqrt{E}y)C_2(x),$$
where  $C_1(x)=\frac{\mu}{a \alpha}J\phi_x(x,0)$ and
$C_2(x)=\frac{1}{\sqrt E}\phi_y(x,0)$. It is clear that
$(C_1(x),C_2(x))=0$. In addition, (\ref{Frenetreal}) and
(\ref{ctes}) imply that $C_1'(x)=\frac{-ia\alpha}{\mu}C_1(x)$ and
$C_2'(x)=\frac{i \mu E}{a \alpha}C_2(x)$.

Choosing in $\c^2$ the unitary reference $\varepsilon_1 =
\phi_x(0,0)/\sqrt \alpha$, $\varepsilon_2 = \phi_y(0,0)/\sqrt
\alpha$ we arrive at
$$ \phi (x,y)=\left(  \frac{i\mu}{a \sqrt{\alpha}} \cosh (\sqrt{E}y) \exp \left(-\frac{i a \alpha}{\mu}x\right),
\frac{\sqrt{\alpha}}{\sqrt E} \sinh (\sqrt{E}y) \exp \left(\frac{i
\mu E}{a \alpha }x\right)\right). $$ Introducing the new parameter
$b:=g(0)>0$, using (\ref{g}) we have $b=\frac{\mu^2}{\alpha a^2}$
and (\ref{ctes}) gives that $1+2ab = \frac{\mu^2 E}{\alpha^2
a^2}$. We observe that if $a<0$ then $0<b<-1/2a$ in order to get
$E>0$. Changing coordinates with $x+iy=\frac{1}{\sqrt E}(s+it)$ we
finally get
$$ \phi(s,t)=\left( \pm i\sqrt{b}\cosh t\,\exp\left(\frac{\mp i\,s}{\sqrt{1+2ab}}\right),
\frac{\sqrt{b}}{\sqrt{1+2ab}}\sinh t\,\exp\left(\pm
i\sqrt{1+2ab}\,s\right)\right),
$$
using the sign $\pm $ according to $a\gtrless 0$. If $a>0$ we put
$b=\frac{\sinh^2 \delta}{2a}$, $\delta
>0$, and this proves part (a); and if $a<0$ we put $b=\frac{\sin^2
\gamma}{-2a}$, $0<\gamma <\pi/2$, and this proves part (b)-(iii).

{\em Case (b): $-\mu^2 \leq E <0$.} In this case remember that
$a<0$ and the reasoning is similar. Using (\ref{Frenetreal}),
(\ref{Frenet1}) and (\ref{Frenet2}), we get
$$ \phi(x,y)= \cos (\sqrt{-E}y)C_1(x)+\sin (\sqrt{-E}y)C_2(x),$$
where $C_1(x)=\frac{\mu}{a \alpha}J\phi_x(x,0)$ and
$C_2(x)=\frac{1}{\sqrt{-E}}\phi_y(x,0)$. 
Again (\ref{Frenetreal}) and (\ref{ctes}) imply that
$C_1'(x)=\frac{-ia\alpha}{\mu}C_1(x)$, $C_2'(x)=\frac{i \mu E}{a
\alpha}C_2(x)$. Hence:
$$ \phi (x,y)=\left(  \frac{i\mu}{a \sqrt{\alpha}} \cos (\sqrt{-E}y) \exp \left(\frac{-i a \alpha}{\mu}x \right),
\frac{\sqrt{\alpha}}{\sqrt{-E}} \sin (\sqrt{-E}y) \exp
\left(\frac{i \mu E}{a \alpha }x\right)\right). $$

Introducing the new parameter $c:=g(0)>0$, using (\ref{g}) we also
have $c=\frac{\mu^2}{\alpha a^2}$ and now (\ref{ctes}) gives that
$-1-2ac = \frac{-\mu^2 E}{\alpha^2 a^2}$ and  observe that $-\mu^2
\leq E <0$  only implies that $c>-1/2a$. Using the coordinates
$x+iy=\frac{1}{\sqrt{-E}}(t+is)$, we finally get
$$ \phi(t,s)=\left(-i\sqrt{c}\cos
s\exp\left(\frac{i\,t}{\sqrt{-\!1-\!2ac}}\right),\frac{\sqrt{c}}{\sqrt{-\!1-\!2ac}}\sin
s\exp\left(i\sqrt{-\!1-\!2ac}\,t\right)\right), $$ where
$\frac{-1}{2a}<c$. Finally we put $c=\frac{\cosh^2 \nu}{-2a}$,
$\nu >0$, and this proves part (b)-(iv).
\end{proof}

\begin{corollary}
The tori $\mathcal{T}_{m,n}$ (described in Proposition~\ref{ex2})
are the only compact orientable Hamiltonian stationary
self-similar solutions for mean curvature flow in complex
Euclidean plane.
\end{corollary}


\end{document}